\documentclass[14pt]{amsart}
\usepackage{amscd,amsthm,amsmath,amssymb}
\usepackage[dvips]{graphicx}
\usepackage[matrix,arrow]{xy}

\makeatletter\@addtoreset{equation}{section}\makeatother

\makeatletter\@addtoreset{subsection}{equation}\makeatother

\newcommand{\p}{\mathbb{P}}
\newcommand{\cel}{\mathbb{Z}}

\newcommand{\com}{\mathbb{C}}

\newcommand{\rea}{\mathbb{R}}

\newcommand{\map}{\longrightarrow}

\newcommand{\vol}{\mathrm{Vol}}

\newtheorem{theorem}[equation]{Theorem}
\newtheorem{prop}[equation]{Proposition}
\newtheorem{lemma}[equation]{Lemma}
\newtheorem{cor}[equation]{Corollary}

\newtheorem{defi}[equation]{Definition}

\theoremstyle{remark}
\newtheorem{remark}[equation]{Remark}

\newtheorem{ex}[equation]{Example}

\thanks{{\it MS 2020 classification}: 14E05, 14M10, 28D20}


\begin{document}

\title{One numerical obstruction for rational maps between hypersurfaces}

\author{Ilya Karzhemanov}
\address{\newline{\normalsize Laboratory of AGHA, Moscow Institute of Physics and Technology, 9 Institutskiy per., Dolgoprudny,
Moscow Region, 141701, Russia}
\newline{\it E-mail address}: karzhemanov.iv@mipt.ru}

\sloppy

\maketitle

\begin{abstract}
Given a rational dominant map $\phi: Y \dashrightarrow X$ between
two generic hypersurfaces $Y,X \subset \p^n$ of dimension $\ge 3$,
we prove (under an addition assumption on $\phi$) a
``\,Noether\,--\,Fano type\,'' inequality $m_Y \ge m_X$ for
certain (effectively computed) numerical invariants of $Y$ and
$X$.
\end{abstract}

\bigskip

\section{Introduction}
\label{section:s-0}

\refstepcounter{equation}
\subsection{Set\,-\,up}
\label{subsection:s-0-1}

Let $X \subset \p^n$ be a smooth hypersurface over $\com$ given by
an equation $f = 0$ in some projective coordinates $x_0, \ldots,
x_{n}$ on $\p^n$. Identify $x_i$ with a basis of $H^0(X,L)$ for $L
:= \mathcal{O}(1)$. We will assume in what follows that $n \ge 4$.
In particular, given two such hypersurfaces $X$ and $Y$, we have
$\text{Pic}\,X = \text{Pic}\,Y = \cel \cdot L$ (Lefschetz), so
that any rational map $\phi: Y \dashrightarrow X$ is induced by a
self\,-\,map of $\p^n$.

\begin{defi}
\label{theorem:symp} Call $\phi$ \emph{symplectic} if the
corresponding map $\p^n \dashrightarrow \p^n$ preserves, up to a
constant, the $2$\,-\,form $\displaystyle\sum_{i = 1}^n
\frac{dx_i}{x_i} \wedge \frac{d\bar{x}_i}{\bar{x}_i}$ (cf.
{\ref{subsection:s-1-1}} below). Also, call $X$
\emph{symplectically unirational}, if $Y := \p^{n-1} = (x_{n} =
0)$ and $\phi$ is symplectic.
\end{defi}

\begin{ex}
\label{example:m-t-p-n-a} Take $X = Y = (x_n = 0) \simeq \p^{n-1}$
and $\phi$ to be the \emph{Frobenius morphism} $x_i \mapsto
x_i^d$, $0 \le i \le n - 1$, for some integer $d \ge 2$. This
$\phi$ is easily seen to be symplectic.
\end{ex}

Fix a point $o \in X$ and consider the blowup $\sigma:
\widetilde{X} \map X$ of $o$. Let $\Sigma := \sigma^{-1}(o)$ be
the exceptional divisor of $\sigma$. Define the quantity (cf.
\cite[Definition 1.6]{i-i})
$$
m_X(L;o) := \sup\left\{\varepsilon \in \mathbb{Q} \ : \ \text{the
linear system} \ |N(\sigma^*L - \varepsilon \Sigma)| \ \text{is
mobile for} \ N \gg 1 \right\}.\footnote{~Recall that
``\,mobile\,'' means no divisorial component in the base locus.}
$$
It is called the \emph{mobility threshold} (of $X$) and was first
introduced by A. Corti in \cite{questions-in-bir-geom} (we will
sometimes write simply $m_X$ when the point $o$ is irrelevant).

\begin{ex}
\label{example:m-t-p-n} Let $f := x_{n}$ and hence $X \simeq \p^{n
- 1}$ is a projective subspace. Then we get $m_X(L;o) = 1$ because
the lines passing through $o$ sweep out a divisor (in particular,
the above $\varepsilon$ is always $\le 1$, while the opposite
inequality is clear --- consider the projection $X \dashrightarrow
\p^{n - 2}$ from $o$ to obtain $\varepsilon = 1$).
\end{ex}

In this paper, we will be mostly using the following equivalent
definition of $m_X$ (although the first one is more suitable for
computations):
$$
m_X(L;o) := \sup\, \frac{\text{mult}_o\,\mathcal{M}}{N},
$$
where sup is taken over all $N \ge 1$ and mobile linear systems
$\mathcal{M} \subseteq |L^{\otimes N}|$.

Let us assume from now on that $X = (f = 0)$ and $Y$ are
\emph{generic}. Here is our main result:

\begin{theorem}
\label{theorem:main} If $\phi$ is symplectic, then there exist
points $o \in X$, $o' \in Y$ such that $m_Y(L;o') \ge m_X(L;o)$.
\end{theorem}

Theorem~\ref{theorem:main} implies in particular that $1 = m_Y \ge
m_X$ for $Y := \p^{n - 1}$ and symplectically unirational $X$ (cf.
Corollary~\ref{theorem:uni} below). Note however that similar
estimate \emph{does not} hold for an arbitrary unirational $X$
(see {\ref{subsection:s-3-2}}).

\refstepcounter{equation}
\subsection{Discussion}
\label{subsection:s-0-2}

The main idea behind the proof of Theorem~\ref{theorem:main} is
that birational invariants of $X$ appear from a (hidden)
\emph{hyperbolic structure} on hypersurfaces. Namely, we employ
the so\,-\,called \emph{pairs\,-\,of\,-\,pants decomposition}
$\Pi$ from \cite{grisha}, which we recall briefly in
Section~\ref{section:s-1}. This brings further an analogy with
hyperbolic manifolds (especially surfaces and $3$\,-\,folds) and
their geometric invariants --- most common ones being various
types of \emph{volumes}.

This includes, as a basic example, the Euler characteristic of
Riemann surface $M$. A finer invariant is the so\,-\,called
\emph{conformal volume} $V_c(M)$. One can show that $2V_c(M)\ge
\lambda_1\text{Vol}(M)$ for the first Laplacian eigenvalue
$\lambda_1$ of $M$ (see \cite[Theorem 1]{li-yau} and corollaries
thereof). This was used, for instance, to obtain obstructions for
existence of maps between Riemann surfaces (see e.\,g. the proof
of the \emph{Surface Coverings Theorem} in \cite[\S
4]{gromov-metric-invariants} or that of \cite[Theorem
$2.\mathrm{A}_1$]{gromov-spectral-geometry}).

Further, if $M$ is a $d$\,-\,dimensional closed oriented
hyperbolic manifold, then there is a \emph{Gromov's invariant}
$\|[M]\|$ (see \cite[Chapter 6]{thurs}). It is defined in terms of
certain (probability) measures on $M$ and coincides with
$C_d\text{Vol}(M)$ for some absolute constant $C_d$. The most
fundamental property of this invariant (used in the proof of
Mostow's rigidity theorem for example) is that for a map $M_1 \map
M_2$ between two hyperbolic $M_i$ one has $\|[M_1]\| \ge
\|[M_2]\|$. The latter inequality and a close similarity between
the definitions of $V_c(~)$, $\|[~]\|$, etc. and that of $m_X$
(cf. {\ref{subsection:s-0-1}} and {\ref{subsection:s-1-2}} below)
have motivated our approach towards the proof of
Theorem~\ref{theorem:main}.

Namely, on replacing $X$ by the complex $\Pi$ mentioned above, we
recast $m_X$ in ``\,probabilistic\,'' terms (see
Section~\ref{section:s-2}). The argument here is an instance of
the \emph{(Bernoulli) law of large numbers} and allows one to give
a conceptual explanation for the estimate $m_Y \ge m_X$ (cf.
Remark~\ref{remark:ent-law-of-l-num}). On the other hand, results
in {\ref{subsection:s-1-2}} and {\ref{subsection:s-1-3}}, together
with initial definition of $m_X$ in {\ref{subsection:s-0-1}},
yield algebro\,-\,geometric applications (see
Section~\ref{section:s-3}).

Such line of thought --- associating $X \rightsquigarrow \Pi$ and
extracting geometric properties of $X$ from combinatorics of $\Pi$
(and vice versa) --- is not new. Classical case includes the
\emph{Brunn\,--\,Minkowski inequality} (see e.\,g.
\cite{gromov-b-m-ineq}). In a modern context (including the
\emph{mirror symmetry}) this viewpoint appears in \cite{gam-shen}
and \cite{ruddat-el-al} for instance. Finally, we mention the
``\,motivic\,'' part of the story, when one assigns to $X$ its
\emph{stable birational volume} $[X]_{\text{sb}}$ (see e.\,g.
\cite{nic-ott}) or its class $[X]_{\mathcal{K}}$ in the
\emph{connective $K$\,-\,theory} (see \cite{kirill}). In the
latter case, given $\phi: Y \dashrightarrow X$ as above, the
\emph{degree formula} of \cite{kirill} relating the classes
$[Y]_{\mathcal{K}}$, $[X]_{\mathcal{K}}$ may be considered as a
vast generalization of \emph{Noether\,--\,Fano inequality} (see
e.\,g. \cite[Proposition 2]{isk}) and was another motivation for
our Theorem~\ref{theorem:main}.

\bigskip

\section{Preliminaries}
\label{section:s-1}

\refstepcounter{equation}
\subsection{Pairs\,-\,of\,-\,pants complex}
\label{subsection:s-1-1}

Consider the intersection
$$
X^0 := X \cap \bigcap_{i = 0}^n (x_i \ne 0)
$$
of $X$ with the torus $(\com^*)^n \subset \p^n$ equipped with
(affine) coordinates $x_1, \ldots, x_n$. We may identify $f$ with
the Laurent polynomial defining $X^0$:
\begin{equation}
\label{eq-f} f = \sum_{\scriptscriptstyle\tiny j \in \Delta \cap
\cel^{n}}a_jt^{-v(j)}x^j,
\end{equation}
where $\Delta \subset \rea^{n}$ is the \emph{Newton polyhedron} of
$f$, $v: \Delta \cap \cel^{n} \map \rea$ is a piecewise affine
function, $a_j \in \com$ and $t > 0$ is a real parameter (cf.
\cite[{\it 6.4}]{grisha}).

Let us recall the \emph{balanced maximal dual polyhedral
$\Delta$\,-\,complex} $\Pi$ associated with $X$. It corresponds to
certain (dual) simplicial lattice subdivision of $\Delta$
associated with the corner locus of the Legendre transform of $v$;
$\Pi$ may also be identified with its moment image (see \cite[{\it
2.3} and Proposition 2.4]{grisha}). Here are some properties of
$\Pi$ we will use below:

\begin{itemize}

\item there exists a smooth $T^{n - 1}$\,-\,equivariant map (\emph{stratified $T^{n - 1}$-fibration}) $\pi: X \map
\Pi$, where $T^{n - 1} := (S^1)^{n - 1}$ consists of the arguments
of $x_i$, so that for any open lattice simplex $\Pi_0 \subset \Pi$
the preimage $\pi^{-1}(\Pi_0)$ is an \emph{open
pair\,-\,of\,-\,pants}, symplectomorphic to $H^{\circ} :=
(\displaystyle\sum_{i = 1}^n x_i = 1) \subset (\com^*)^{n}$
equipped with the $2$\,-\,form $\Omega :=
\displaystyle\frac{1}{2\sqrt{-1}}\sum_{i = 1}^n \frac{dx_i}{x_i}
\wedge \frac{d\bar{x}_i}{\bar{x}_i}$;

\smallskip

\item in fact, $X$ is glued out of the \emph{tailored} (or \emph{localized})
pants $Q^{n - 1}$ (see \cite[{\it 6.6} and Proposition
4.6]{grisha}), isotopic to $H^{\circ}$, so that $\pi$ is a
deformation retraction under the Liouville flow associated with
$\Omega$, and $\pi_*\Omega^{n - 1}$ induces the Euclidean measure
on $\Pi$;

\smallskip

\item $\Pi$ is also obtained via the \emph{tropical degeneration}
(compatibly with $\pi$): recall that $f$ depends on $t$ (see
\eqref{eq-f}) and let $t \to 0$ (for each $Q^{n - 1}$) in the
preceding constructions
--- $\Pi$ is then a Gromov\,--\,Hausdorff limit of amoebas $\mathcal{A}_t(X_0)$ (see
\cite[{\it 6.4}]{grisha}).
\end{itemize}

\refstepcounter{equation}
\subsection{Atomic (probability) measure}
\label{subsection:s-1-2}

Choose a global section $s \in H^0(X,L)\setminus{\{0\}}$ and a
point $o \in X$. We are going to construct a measure $d\mu_{o,s}$
on $\Pi$ so that $\vol\,\big(\Pi\big) = \text{mult}_{o}\left\{s =
0\right\}$ with respect to it and $d\mu_{o,s}$ is supported at the
point $\pi(o)$.\footnote{~All considerations below apply
literarily to any $L^{\otimes N}$, $N \ge 1$, in place of $L$.}

Identify $s$ with a holomorphic function on a complex neighborhood
$U \subset X$ of $o$ and consider the $(1,1)$\,-\,current $\tau :=
\displaystyle\frac{\sqrt{-1}}{2\pi}\,\partial\overline{\partial}\log|s|$.
Then $\tau$ acts on the $L^1$\,-\,forms on $U$ of degree $2n - 2$
via (Poincar\'e\,--\,Lelong)
$$
\omega \mapsto \tau(\omega) = \displaystyle\int_{s\, =\, 0}
\omega.
$$
In particular, if $\omega$ is the $(n-1)$\,-\,st power of the
Fubini\,--\,Studi form on $\p^n$ restricted to $U$, one may regard
$\tau(\omega)$ as the volume of the locus $U \cap (s = 0)$.

Let us assume from now that $U := B_o(r)$ is the Euclidean ball of
radius $r$ centered at $o$. For all $\xi \in \com$, $|\xi| < 1$,
we consider the dilations $z \mapsto \xi z$ of $U$ and the family
of pushed\,-\,forward measures $\xi_*(dm)$.

The following lemma is standard (cf. Remark~\ref{remark:con-der}
and Proposition~\ref{theorem:c-ent-m} below):

\begin{lemma}
\label{1-1-cur-lim} The limit measure
$$
\displaystyle\frac{1}{(r\xi)^{2n - 2}}\lim_{\xi \to\, 0}\
\xi_*(dm) := dm_{o,s}
$$
exists and $\displaystyle\int_U dm_{o,s} =
\mathrm{mult}_{o}\left\{s = 0\right\}$.
\end{lemma}

\begin{proof}
Indeed, $\vol\,\big(\{s = 0\} \cap U\big)$ with respect to
$\displaystyle\frac{1}{(r\xi)^{2n - 2}}\ \xi_*(dm)$ tends to
$\text{mult}_{o}\left\{s = 0\right\}$, as $\xi \to 0$. This
implies that the limit of measures exists.
\end{proof}

We may assume that for $U = B_o(r)$ the radius $r \to 0$ as $t \to
0$ (cf. \eqref{eq-f}). Then it follows from
{\ref{subsection:s-1-1}} and Lemma~\ref{1-1-cur-lim} that
\begin{eqnarray}
\nonumber \pi(U) = \text{simplicial complex}\, \Pi_{0} = \text{Gromov\,--\,Hausdorff limit of}\,\mathcal{A}_t(U),\\
\label{prop-gh} s = \pi^*\ell \ \text{for a piecewise linear function}\,\ell\,\text{on}\,\Pi_{0},\\
\nonumber d\mu_{o,s,\scriptscriptstyle\tiny\Pi_{0}} :=
\pi_*dm_{o,s} = \text{measure on}\,
\Pi_{0}\, \text{supported at}\,\pi(o)\,\text{and such that}\\
\nonumber \vol\,\big(\Pi_{0}\big) =
\int_{\scriptscriptstyle\tiny\Pi_{0}}
d\mu_{o,s,\scriptscriptstyle\tiny\Pi_{0}} =
\text{mult}_{o}\left\{s = 0\right\};
\end{eqnarray}
note also that $\pi(o)$ belongs to the corner locus of $\ell$.

Finally, $d\mu_{o,s,\scriptscriptstyle\tiny\Pi_{0}}$ induces a
measure $d\mu_{o,s}$ on $\Pi\supseteq\Pi_{0}$ in the obvious way,
which concludes the construction.

\begin{remark}
\label{remark:con-der} The atomic measure $d\mu_{o,s}$ is an
instance of the \emph{convexly derived measure} from
\cite{gromov-waists} (one may also treat the above ``\,density
function\,'' $\ell$ as a discrete version of the Hessian of
$\displaystyle\frac{\sqrt{-1}}{2\pi}\,\log|s|$). The ``\,mass
concentration\,'' concept of \cite{gromov-waists} will be used
further to obtain intrinsic (bi)\,rational invariants of $X$.
\end{remark}

\refstepcounter{equation}
\subsection{Rational maps: tropicalization}
\label{subsection:s-1-3}

Let $Y \subset \p^n$ be another hypersurface, similar to $X$, with
the maximal dual complex $\Pi^Y$, projection $\pi^{Y}: Y \map
\Pi^Y$, etc. defined verbatim for $Y$. Assume also that there
exists a rational symplectic map $\phi: Y \dashrightarrow X$. Then
the constructions in {\ref{subsection:s-1-1}} and
{\ref{subsection:s-1-2}} yield a map $\Phi: \Pi^Y \map \Pi^X$,
given by some PL functions with
$\cel$\,-\,coefficients,\footnote{~We use the notation $\Pi^X = :
\Pi$ and $\pi^X := \pi$ in what follows.} so that the following
diagram commutes:
\begin{equation}
\label{com-diag} \xymatrix{
Y\ar@{->}[d]_{\pi^Y}\ar@{-->}[r]^{\phi}&X\ar@{->}[d]^{\pi^X}\\
\Pi^Y\ar@{->}[r]^{\Phi}&\Pi^X.}
\end{equation}
Note that $\Phi$ \emph{need not} necessarily be a map of
simplicial complexes.

The following lemma describes $\Phi$ as a map of \emph{measure
spaces}:

\begin{lemma}
\label{theorem:meas-pre-under-f-star} There exists a positive
number $\delta_{\phi} \in \cel$, depending only on $\phi$, such
that $\Phi^*\pi^X_*\Omega^{n - 1} = \delta_{\phi}\pi^Y_*\Omega^{n
- 1}$.\footnote{~Here $\Phi^*$ is defined with respect to the
(limiting) affine structure on $\Pi$ induced from the complex one
on $X$ (cf. {\ref{subsection:s-1-1}}).}
\end{lemma}

\begin{proof}
We have
$$
\pi^Y_*\phi^*\Omega^{n - 1} = \delta_{\phi}\pi^Y_*\Omega^{n - 1}
$$
for some real $\delta_{\phi} > 0$ (cf.
Definition~\ref{theorem:symp}). Now, since $\phi$ is the
restriction of a rational self\,-\,map of $\p^n$ (see
{\ref{subsection:s-0-1}}), it follows from
{\ref{subsection:s-1-1}} and \eqref{com-diag} that
$\pi^Y_*\Omega^{n - 1}$ (resp. $\Phi^*\pi^X_*\Omega^{n - 1}$)
coincides with the measure induced by the standard one $dy_1
\wedge \ldots \wedge dy_n$ on $\rea^n$ (resp. by $dl_1 \wedge
\ldots \wedge dl_n$ for some piecewise linear functions $l_i =
l_i(y)$ with $\cel$\,-\,coefficients). It remains to observe that
$dl_1 \wedge \ldots \wedge dl_n = \delta_{\phi}dy_1 \wedge \ldots
\wedge dy_n$ by construction.
\end{proof}

\bigskip

\section{Proof of Theorem~\ref{theorem:main}}
\label{section:s-2}

\refstepcounter{equation}
\subsection{The entropy}
\label{subsection:s-2-1}

Let $\Pi \subset \rea^n$ be a simplicial complex with the standard
Borel measure $d\mu$. Fix some real number $M > 0$ and consider
various measures $d\mu_{\ell}$ on $\Pi$, supported at the corner
locus of PL functions $\ell$, such that $\displaystyle \int_{\Pi}
d\mu_{\ell} \le M$. Let $\mathcal{S} := \mathcal{S}(\Pi,M)$ be the
set of all such measures (aka functions).

Further, given an integer $N > 0$ the measure space $(\Pi,Nd\mu)
=: \Pi_N$ may be regarded as $\Pi \subset \rea^{Nn}$, embedded
diagonally, with the measure being $d\mu^N := \displaystyle\sum_{i
= 1}^N\pi_i^*d\mu$ for the $i^{\text{th}}$ factor projections
$\pi_i: \rea^{Nn} \map \rea^n$. The affine structure on $\Pi_N$ is
defined by the functions $\displaystyle\sum_{i =
1}^N\pi_i^*\ell_i$ for various PL $\ell_i$. Note that
\begin{equation}
\nonumber\frac{1}{N}\int_{\Pi_N}\sum_{i =
1}^N\pi_i^*(d\mu_{\ell_i}) \le M,
\end{equation}
i.\,e. $\displaystyle\frac{1}{N}\sum_{i =
1}^N\pi_i^*(d\mu_{\ell_i}) \in \mathcal{S}$, provided
$d\mu_{\ell_i} \in \mathcal{S}$ for all $1 \le i \le n$.

Define the measures $d\mu^N_{\ell}$ on $\Pi_N \subset \rea^{Nn}$
and the set $\mathcal{S}(\Pi_N,M) \ni
\displaystyle\frac{1}{N}\,d\mu^N_{\ell}$ similarly as above. Let
also
\begin{equation}
\label{def-c} C := \sup_{\ell\in\mathcal{S}(\Pi_N,M),\, N}
\displaystyle \frac{1}{N} \int_{\Pi_N} d\mu^N_{\ell}.
\end{equation}

\begin{prop}
\label{theorem:c-ent-m} There exists a number
$\mathrm{ent}(d\mu,\mathcal{S}) < \infty$, \emph{depending only on
$d\mu$ and $\mathcal{S}$}, such that $C =
\mathrm{ent}(d\mu,\mathcal{S})M$.
\end{prop}

\begin{proof}
After normalizing we may assume that $M = 1$. Let us also assume
for transparency that $\Pi$ is a simplex.

All measures $\displaystyle\frac{1}{N}\,d\mu^N_{\ell}$ can be
identified with points (mass centers) in the dual simplex
$\Pi^*\subset\rea^n$ (compare with the proof of
\cite[4.4.A]{gromov-waists}). Let
$\mathcal{H}_{\mu,\mathcal{S}}\subseteq\Pi^*$ be the convex hull
of this set. Then
$$
\displaystyle \int_{\Pi}\bullet:\ \mathcal{H}_{\mu,\mathcal{S}}
\longrightarrow \rea_{\ge 0}
$$
is a \emph{bounded} ($\le 1$) linear functional. By definition we
obtain $C =
\displaystyle\max_{\mathcal{H}_{\mu,\mathcal{S}}}\displaystyle
\int_{\Pi}\bullet =: \text{ent}(d\mu,\mathcal{S})$ and the result
follows.
\end{proof}

\begin{remark}
\label{remark:ent-law-of-l-num} The constant $C = C^X$ resembles
the value of \emph{logarithmic rate decay function} at $d\mu$ (see
e.\,g. \cite[Lecture 4]{gromov-prob}). This suggests $C$ to be
equal the ``\,Boltzmann entropy\,'' and the estimate $C^Y \ge C^X$
in the setting of {\ref{subsection:s-1-3}} (compare with
\cite[p.\,7]{gromov-ent}). In fact, taking $d\mu_{\ell} :=
d\mu_{o,s}$ for various $s$ as in {\ref{subsection:s-1-2}}, we
will apply this probabilistic reasoning to (birational) geometry
of $X$ (see below).
\end{remark}

\refstepcounter{equation}
\subsection{The estimate}
\label{subsection:s-2-2}

Let $\Phi: \Pi^Y \map \Pi^X$ be as in {\ref{subsection:s-1-3}}.
Although $\Phi$ need not preserve the simplicial structures, we
still can find a pair of $k$\,-\,simplices $\Pi_0^X \subseteq
\Pi^X$ and $\Pi_0^Y \subseteq \Pi^Y$, $1 \le k \le n$, such that
$\Phi(\Pi_0^Y) = \Pi_0^X$.

Identify both $\Pi_0^X$ and $\Pi_0^Y$ with a simplex $\Pi$,
carrying two (Borel) measures $d\mu$ and $\delta_{\phi}\,d\mu$,
induced by $\pi^Y_*\Omega^{n - 1}$ and $\Phi^*\pi^X_*\Omega^{n -
1}$, respectively (see Lemma~\ref{theorem:meas-pre-under-f-star}).

Let us assume from now on that $\mathcal{S} := \mathcal{S}_X$
consists of PL functions $\ell$, obtained from various sections $s
= \pi^*\ell \in \mathcal{M}$ and mobile linear systems
$\mathcal{M} \subseteq |L^{\otimes N}|$, so that $d\mu_{\ell} =
\displaystyle\frac{1}{N}\,d\mu_{o,s}$ for some $o \in X$
satisfying $\pi(o) \in \Pi$ (see {\ref{subsection:s-2-1}} and
\eqref{prop-gh}). It follows from \eqref{prop-gh} that $M$ in
{\ref{subsection:s-2-1}} can be assumed to coincide with the
mobility threshold $M^X := m_X(L;o)$ (cf.
{\ref{subsection:s-0-1}}). Same considerations apply to $Y$, with
$\mathcal{S}_Y$, $M^Y := m_Y(L;o)$, etc.

\begin{lemma}
\label{theorem:ent-1} In the previous setting, we have
$\mathrm{ent}(d\mu,\mathcal{S}_X) = 1$, and similarly for
$\mathcal{S}_Y$.
\end{lemma}

\begin{proof}
This follows from \eqref{def-c} (cf.
Proposition~\ref{theorem:c-ent-m}), definition of $m_X$ (cf.
\eqref{prop-gh}), and the fact that $\pi^X_*\Omega^{n - 1} = d\mu$
on $\Pi = \Pi_0^X$ (see {\ref{subsection:s-1-1}}).
\end{proof}

\begin{prop}
\label{theorem:phi-s-q-m} For every $\ell\in\mathcal{S}(\Pi,M^X)$,
we have $d\mu_{\Phi^*\ell} = d\mu_{\tilde{\ell}}$, where
$\tilde{\ell}\in\mathcal{S}(\Pi,M^Y)$.
\end{prop}

\begin{proof}
It follows from \eqref{prop-gh} and Lemma~\ref{1-1-cur-lim} that
$$
\int_{\Pi} d\mu_{\ell} =
\displaystyle\frac{1}{N}\,\int_{\Pi}d\mu_{o,s} =
\displaystyle\frac{1}{N}\,\int_U dm_{o,s} =
\displaystyle\frac{1}{N}\,\int_{U \setminus Z} dm_{o,s}
$$
for any closed subset $Z \subsetneq U$. Recall that the
\emph{rational transform} $\phi^{-1}_*s$ is naturally defined as a
member of the mobile linear system $\phi^{-1}_*\mathcal{M}$. In
particular, if $\phi$ is a morphism over $U \setminus Z$, then
$$
\int_{\phi^{-1}(U \setminus Z)} dm_{o,\phi_*^{-1}s} =
\text{mult}_{o}\left\{\phi_*^{-1}s = 0\right\}.\footnote{~There is
a slight abuse of notation here --- $o$ denotes a point in both
$X$ and $Y$.}
$$
This $\phi_*^{-1}s$ defines a PL function $\tilde{\ell}$ as
earlier and we have
$$
\int_{\Pi} d\mu_{\Phi^*\ell} = \int_{\Pi} d\mu_{\tilde{\ell}}
$$
(cf. \eqref{com-diag}). The identity $d\mu_{\Phi^*\ell} =
d\mu_{\tilde{\ell}}$ follows and
$\tilde{\ell}\in\mathcal{S}(\Pi,M^Y)$ by construction.
\end{proof}

Let $C := M^Y$ be as in Proposition~\ref{theorem:c-ent-m}
($\text{ent}(d\mu,\mathcal{S}) = 1$ by Lemma~\ref{theorem:ent-1})
and $\delta_{\phi}$ as in
Lemma~\ref{theorem:meas-pre-under-f-star}. Then it follows from
Proposition~\ref{theorem:phi-s-q-m} (cf.
Remark~\ref{remark:ent-law-of-l-num}) that
$$
C\delta_{\phi} = \sup_{\tilde{\ell}\in\mathcal{S}(\Pi_N,M^Y),\, N}
\displaystyle \frac{1}{N} \int_{\Pi_N}
\delta_{\phi}d\mu^N_{\tilde{\ell}} \ge
\sup_{\ell\in\Phi^*\mathcal{S}(\Pi_N,M^X),\, N} \displaystyle
\frac{1}{N} \int_{\Pi_N} \delta_{\phi}d\mu_{\Phi^*\ell}^N =
$$
$$
= \sup_{\ell\in\mathcal{S}(\Pi_N,M^X),\, N} \displaystyle
\frac{1}{N} \int_{\Pi_N} \Phi^*d\mu_{\ell}^N =
\text{ent}(\delta_{\phi}d\mu,\mathcal{S})M^X
$$
(the last equality is due to the projection formula
$\Phi_*\Phi^*d\mu = \delta_{\phi}d\mu$ and the change of variables
in $\int$). Finally, since
$\text{ent}(\delta_{\phi}d\mu,\mathcal{S}) =
\delta_{\phi}\text{ent}(d\mu,\mathcal{S})$, we conclude that $M^Y
\ge M^X$.

\bigskip

\section{Some examples and applications}
\label{section:s-3}

\refstepcounter{equation}
\subsection{Soft}
\label{subsection:s-3-1}

Setting $Y := \p^{n-1} = (x_{n} = 0)$ we arrive at the following
immediate

\begin{cor}
\label{theorem:uni} Suppose $X$ in Theorem~\ref{theorem:main} is
symplectically unirational. Then there exists a point $o \in X$
such that $m_X(L;o) = 1$.
\end{cor}

\begin{proof}
It suffices to prove that $m_X \ge 1$. This is done by considering
the projection $X \dashrightarrow \p^{n-1}$ from $o$ and observing
that the linear system $|\sigma^*L - a\Sigma|$ is mobile for some
$a \ge 1$ (cf. {\ref{subsection:s-0-1}}).
\end{proof}

Suppose $X$ is a \emph{quadric}. Although we do not know whether
$X$ is symplectically unirational (cf.
Example~\ref{example:m-t-p-n-a}), it is obviously rational, and
Corollary~\ref{theorem:uni} confirms that $m_X = 1$ in this case
(the latter equality can actually be proved directly by
considering families of lines on $X$ as in
Example~\ref{example:m-t-p-n}).

\begin{remark}
\label{remark:bir-hyp} It would be interesting two find out
whether any \emph{birationally isomorphic} hypersurfaces $X$ and
$Y$ as in Theorem~\ref{theorem:main} always have $m_X(L;o) =
m_Y(L;o')$ for some points $o$ and $o'$. It should also be
possible to generalize all our considerations to the case of
\emph{any smooth} $X$ and $Y$.\footnote{~Recall that initially in
{\ref{subsection:s-0-1}} these hypersurfaces were assumed
\emph{generic}.}
\end{remark}

Let us proceed with non\,-\,trivial examples distinguishing
ordinary unirationality from the symplectic one.

\refstepcounter{equation}
\subsection{Hard}
\label{subsection:s-3-2}

Suppose $\deg f = 3$ (i.\,e. $X$ is a \emph{cubic}). It is a
classical fact that $X$ is unirational (see e.\,g. \cite[Chapter
3, Corollary 1.18]{daniel}). Fix a point $o \in X$. We may assume
that $o = [1:0:\ldots :0]$, and hence $f = q_1 + q_2 + q_3$ in the
affine chart $(x_0 \ne 0)$, where $q_i = q_i(x_1,\ldots,x_n)$ are
forms of degree $i$. Arguing as in \cite[Section 1]{pukh} we
obtain that $q_1$ and $q_2$ are \emph{coprime}. Thus the linear
system $\mathcal{M} \subset |2L|$ spanned by $q_1^2$ and $q_2$ is
mobile. We conclude that $m_X \ge 3/2$, since
$\text{mult}_o\,\mathcal{M} = 3$ (cf. {\ref{subsection:s-0-1}}),
and so $X$ is \emph{not} symplectically unirational by
Corollary~\ref{theorem:uni}.

Now assume only that $X$ is smooth (cf.
Remark~\ref{remark:bir-hyp}). Then it is possible to find a
(\emph{Eckardt}) point $o \in X$ for which $m_X(L;o) = 1$ (see
\cite[Chapter 5]{daniel}). It would be interesting to study
whether such cubics are symplectically unirational.

Further, consider the case $\deg f = 4 = n$, assuming again that
the quartic $X$ is just smooth. Note that it is still unknown
whether any such $X$ is unirational.\footnote{~Although a smooth
quartic hypersurface $X$ \emph{is} unirational when $n \gg 4$ (see
\cite[Corollary 3.8]{hmp}.} Here is a classical unirational
example after Segre (cf. \cite[{\bf 9.2}]{isk-man}):
$$
X = (x_0^4 + x_0x_4^3 + x_1^4 - 6x_1^2x_2^2 + x_2^4 + x_3^4 +
x_3^3x_4 = 0).
$$
We claim that $m_X(L;o) = 1$ for some $o \in X$. Indeed, take the
hyperplane $\Pi := (x_1 - \alpha x_2 = 0)$, where $\alpha :=
\sqrt{3 + 2\sqrt{2}}$. Then $X \cap \Pi$ is a \emph{cone} in
$\p^3$ given by the equation $x_0^4 + x_0x_4^3 + x_3^4 + x_3^3x_4
= 0$. We take $o$ to be the vertex of this cone.

At the same time, if $X$ is generic, then one can show that $m_X
\ge 3/2$ by exactly the same argument as in the cubic
case.\footnote{~It is proved in \cite[{\bf A.24}]{ilya} that in
fact $m_X = 3/2$.} Thus again symplectic version of the
unirationality problem for $X$ is settled here.

\bigskip

\thanks{{\bf Acknowledgments.}
I am grateful to John Christian Ottem for valuable comments. The
work was carried out at the Center for Pure Mathematics (MIPT) and
supported by the State assignment project FSMG\,-\,202\,-\,0013.


\bigskip

\begin{thebibliography}{21}

\bibitem{i-i} I. Cheltsov\ and\ I. Karzhemanov, Halphen pencils
on quartic threefolds, Adv. Math., {\bf 223} (2010), 594 -- 618.

\smallskip

\bibitem{gam-shen}
B. Gammage\ and\ V. Shende, Mirror symmetry for very affine
hypersurfaces, Acta Mathematica, {\bf 229} (2022), no.~2, 287 --
346.

\smallskip

\bibitem{gromov-b-m-ineq}
M. Gromov, Convex sets and K\"ahler manifolds, in {\it Advances in
differential geometry and topology}, 1 -- 38, World Sci. Publ.,
Teaneck, NJ.

\smallskip

\bibitem{gromov-ent}
M. Gromov, In a Search for a Structure, Part 1: On Entropy,
https://www.ihes.fr/~gromov/category/expository/.

\smallskip

\bibitem{gromov-waists}
M. Gromov, Isoperimetry of waists and concentration of maps, Geom.
Funct. Anal., {\bf 13} (2003), no.~1, 178 -- 215.

\smallskip

\bibitem{gromov-metric-invariants}
M. Gromov, Metric invariants of K\"ahler manifolds, in {\it
Differential geometry and topology (Alghero, 1992)}, 90 -- 116,
World Sci. Publ., River Edge, NJ.

\smallskip

\bibitem{gromov-prob}
M. Gromov, Six lectures on Probability, Symmetry, Linearity,
https://www.ihes.fr/~gromov/category/expository/.

\smallskip

\bibitem{gromov-spectral-geometry}
M. Gromov, Spectral geometry of semi-algebraic sets, Ann. Inst.
Fourier (Grenoble), {\bf 42} (1992), no.~1\,-\,2, 249 -- 274.

\smallskip

\bibitem{hmp}
J. Harris, B. Mazur\ and\ R. Pandharipande, Hypersurfaces of low
degree, Duke Math. J., {\bf 95} (1998), no.~1, 125 -- 160.

\smallskip

\bibitem{daniel}
D. Huybrechts, The geometry of cubic hypersurfaces, Lecture notes
at the University of Bonn, 2019, available at
http://www.math.uni-bonn.de/people/huybrech/Notes.pdf.

\smallskip

\bibitem{isk}
V.\,A. Iskovskikh, On the rationality problem for algebraic
threefolds, Proc. Steklov Inst. Math., {\bf 218} (1997), 186 --
227.

\smallskip

\bibitem{isk-man}
V. A. Iskovskih\ and\ Ju. I. Manin, Three\,-\,dimensional quartics
and counterexamples to the L\"{u}roth problem, Mat. Sb. (N.S.),
{\bf 86} (1971), no.~128, 140 -- 166.

\smallskip

\bibitem{ilya}
I. Karzhemanov, Around the uniform rationality, Preprint
arXiv:1812.03427v1 (2018).

\smallskip

\bibitem{li-yau}
P. Li\ and\ S. T. Yau, A new conformal invariant and its
applications to the Willmore conjecture and the first eigenvalue
of compact surfaces, Invent. Math., {\bf 69} (1982), no.~2, 269 --
291.

\smallskip

\bibitem{grisha}
G. Mikhalkin, Decomposition into pairs\,-\,of\,-\,pants for
complex algebraic hypersurfaces, Topology, {\bf 43} (2004), no.~5,
1035 -- 1065.

\smallskip

\bibitem{nic-ott}
J. Nicaise\ and\ J.\,S. Ottem, Tropical degenerations and stable
rationality, Duke Math. J., {\bf 171} (2022), no.~15, 3023 --
3075.

\smallskip

\bibitem{questions-in-bir-geom}
Questions in Algebraic Geometry, MSRI notes, 2009, available at
www-personal.umich.edu/$\sim\,$erman/Papers/Questions2.pdf.

\smallskip

\bibitem{pukh}
A. Pukhlikov, Birational automorphisms of Fano hypersurfaces,
Invent. Math., {\bf 134} (1998), 401 -- 426.

\smallskip

\bibitem{ruddat-el-al}
H. Ruddat\ et al., Skeleta of affine hypersurfaces, Geom. Topol.,
{\bf 18} (2014), no.~3, 1343 -- 1395.

\smallskip

\bibitem{thurs}
W. P. Thurston, {\it Three\,-\,dimensional geometry and topology.
Vol. 1}, Princeton Mathematical Series, 35, Princeton University
Press, Princeton, NJ, 1997.

\smallskip

\bibitem{kirill}
K. Zainoulline, Degree formula for connective $K$\,-\,theory,
Invent. Math., {\bf 179} (2010), 507 -- 522.

\end{thebibliography}
\end{document}